\documentclass[10pt]{my_NSP}

\usepackage{url,floatflt}
\usepackage{helvet,times}
\usepackage{mathptmx,amsmath,amssymb,bm}
\usepackage{float}
\usepackage[bf,hypcap]{caption}

\emergencystretch=\maxdimen
\def\firstpage{1}
\def\thevol{?}
\def\thenumber{?}
\def\theyear{2017}

\begin{document}

\titlefigurecaption{{\large \bf \rm Progress in Fractional Differentiation
and Applications}\\ {\it\small An International Journal}}

\title{On a Fractional Oscillator Equation with Natural Boundary Conditions}

\author{Assia Guezane-Lakoud{$^1$}, Rabah Khaldi{$^1$} and Delfim F. M. Torres{$^{2,3,*}$}} 

\institute{Laboratory of Advanced Materials,
Department of Mathematics, Badji Mokhtar--Annaba University,
P.O. Box 12, 23000 Annaba, Algeria\and
Center for Research and Development 
in Mathematics and Applications (CIDMA),
Department of Mathematics, University of Aveiro,
3810-193 Aveiro, Portugal\and
African Institute for Mathematical Sciences (AIMS-Cameroon),
P. O. Box 608 Limbe, Cameroon}

\titlerunning{On a Fractional Oscillator Equation ...}
\authorrunning{A. Guezane-Lakoud et al. }

\mail{delfim@ua.pt}

\received{19 Dec. 2016}
\revised{24 Jan. 2017}
\accepted{24 Jan. 2017}


\abstracttext{We prove existence of solutions for a
nonlinear fractional oscillator equation 
with both left Riemann--Liouville
and right Caputo fractional derivatives 
subject to natural boundary conditions. 
The proof is based on a transformation of the problem 
into an equivalent lower order fractional boundary value problem 
followed by the use of an upper and lower solutions method. 
To succeed with such approach, we first prove a result 
on the monotonicity of the right Caputo derivative.}

\keywords{Boundary value problems, 
fractional derivatives, upper and lower solutions method, 
existence of solutions, integral equations.\\[2mm]
}

\maketitle


\section{\; Introduction}

Fractional calculus is an interesting field of research due to
its ability to describe memory properties of materials and, 
therefore, providing a better representation of physical models. 
Because of this, the study of nonlinear fractional differential 
equations has attracted a lot of attention and many papers 
and monographs are devoted to the subject \cite{12,13,14}.
Here, we are concerned with the solvability of a nonlinear 
fractional oscillator equation involving both Riemann--Liouville 
and Caputo fractional derivatives with natural boundary conditions:
\begin{equation}
\label{1.1}
\omega^{2}u\left( t\right) 
- {^{C}D_{1^{-}}^{p}} \, D_{0^{+}}^{q}u\left( t\right)
=f\left( t,u\left( t\right) \right),
\quad 0\leq t\leq 1,
\quad \omega \in \mathbb{R},
\quad \omega \neq 0,  
\end{equation}
with the initial condition
\begin{equation}
\label{1.2}
u\left( 0\right) =0 
\end{equation}
and the natural condition (see \cite{2,3})
\begin{equation}
\label{1.3}
D_{0^{+}}^{q}u\left( 1\right) =0, 
\end{equation}
where $0<p,q<1$, $^{C}D_{1^{-}}^{p}$ is the right side Caputo derivative, 
$D_{0^{+}}^{q}$ denotes the left side Riemann--Liouville derivative, $u$ is the unknown
function, and $f\in C\left( \left[ 0,1\right] \times \mathbb{R},\mathbb{R}\right)$.
We denote problem \eqref{1.1}--\eqref{1.3} by $(P_1)$. Note that if $p=q \rightarrow 1$, 
then problem $(P_1)$ is a classical oscillator boundary value problem \cite{1}.

Oscillator equations appear in different fields of science, such as
classical mechanics, electronics, engineering, and fractional calculus,
being a subject of strong current research: see, e.g., 
\cite{MR3499534,MR3556551,MR3514456} and references therein.
Different methods are used to solve such equations, for example, 
by the Laplace transform method or by using numerical methods \cite{MR3320682}. 
Since some phenomena obey an equation of motion with fractional derivatives, 
oscillator equations with fractional derivatives are a particularly interesting 
subject to study \cite{2,3,MR3320682,MR3499534,4,MR2899123}.

Blaszczyk studied numerically the associated linear problem of $(P_1)$
with $f\left( t,u\left( t\right) \right) =Ag\left( t\right)$, see \cite{3}. 
In \cite{2}, Agrawal discussed the relationship between transversality 
and natural boundary conditions in order to solve fractional differential equations.
Moreover, he gave some interesting examples \cite{2}. To the best of our knowledge, 
most works in the literature have studied problem $(P_1)$ only numerically and 
with a term $f$ in the right-hand side of equation \eqref{1.1} that does not depend on $u$. 
Differently, here we study problem $(P_1)$ by the lower and upper solutions method, 
considering a more general situation where the nonlinear term $f$ 
is a function of $u$. This is important since the physical phenomena
described by the differential equations are mainly of nonlinear nature.

The method of upper and lower solutions is an efficient tool in the study of
differential equations \cite{MR3513976}. Indeed, when we apply this method, 
we prove not only existence of solution, but we also get its location between 
the lower and upper solutions. The method was first introduced by Picard in 1893,
later developed by Dragoni, and then becoming a useful tool to prove existence 
of a solution for ordinary as well as fractional differential equations 
\cite{5,6,9,10,11}.

The paper is organized as follows. Section~\ref{sec:2} is devoted to some
definitions on fractional calculus and properties that will be used later. 
We also define the upper and lower solutions for problem $(P_1)$. 
Our results are given in Sections~\ref{sec:3} and \ref{sec:4}.
The main result is Theorem~\ref{thm:mr}, which establishes existence
of solution for problem $(P_1)$. To prove it, we make use of several
auxiliary results. The first of them is given in Section~\ref{sec:3},
where we provide a monotonicity result for the right Caputo derivative.
In Section~\ref{sec:4}, we convert problem $(P_1)$ into an equivalent 
Caputo boundary value problem of order $p$ that, 
under some conditions on the nonlinear term $f$,
is used to prove existence of solutions for problem $(P_1)$ between 
the reversed ordered lower and upper solutions. Moreover, we construct 
explicitly the upper and lower solutions. The new results of the paper
are then illustrated through an example in Section~\ref{sec:5}.


\section{\; Preliminaries}
\label{sec:2}

This section is devoted to recall some essential definitions on fractional
calculus \cite{12,13,14}. We also define some concepts related to upper 
and lower solutions.

\begin{definition} 
\ Let $g$ be a real function defined on $\left[ 0,1\right]$ and 
$\mu>0$. Then the left and right Riemann--Liouville fractional 
integrals of order $\mu$ of $g$ are defined respectively by
\begin{equation*}
I_{0^{+}}^{\mu }g(t)
=\frac{1}{\Gamma \left( \mu \right)}
\int_{0}^{t} \frac{g(s)}{(t-s)^{1-\mu }}ds
\end{equation*}
and
\begin{equation*}
I_{1^{-}}^{\mu }g(t) 
= \frac{1}{\Gamma \left( \mu \right)}
\int_{t}^{1} \frac{g(s)}{(s-t)^{1-\mu }}ds.
\end{equation*}
The left Riemann--Liouville and the right
Caputo fractional derivatives of order $0<\mu <1$ 
of function $g$ are
\begin{equation*}
D_{0^{+}}^{\mu }g(t) 
=\frac{d}{dt}\left( I_{0^{+}}^{1-\mu }g\right) (t)
\end{equation*}
and
\begin{equation*}
^{C}D_{1^{-}}^{\mu }g(t) 
=-I_{1^{-}}^{1-\mu }g^{\prime }(t),
\end{equation*}
respectively.
\end{definition}

With respect to the properties of Riemann--Liouville 
and Caputo fractional derivatives,
we recall here two of them.
Let $0<\mu <1$ and $f\in L_{1}\left[ 0,1\right]$. Then,
\begin{enumerate}
\item\  $I_{0^{+}}^{\mu }D_{0^{+}}^{\mu }f\left( t\right) 
=f\left( t\right)+ct^{\mu -1}$ almost everywhere on $\left[ 0,1\right]$;

\item\ $I_{1^{-}}^{\mu}$$^{C}D_{1^{-}}^{\mu}f\left( t\right) 
=f\left( t\right)-f\left( 1\right)$.
\end{enumerate}

Now, we give the definition of lower and upper solutions 
for problem $(P_1)$. By $AC^{2}\left[ 0,1\right]$ we denote
the following space of functions:
\begin{equation*}
AC^{2}\left[ 0,1\right] :=
\left\{ u\in C^{1}\left[ 0,1\right] \, | \, 
u^{\prime} \text{ is an absolutely continuous function on }
\left[ 0,1\right] \right\}.
\end{equation*}

\begin{definition}
\ Functions $\alpha, \beta \in AC^{2}\left[ 0,1\right]$ are called,
respectively, lower and upper solutions of problem $(P_1)$ if

$\omega^{2}\alpha \left( t\right) 
- ^{C}D_{1^{-}}^{r}D_{0^{+}}^{q}\alpha \left( t\right) 
-f\left( t,\alpha \left( t\right) \right) \leq 0$
for all $t\in \lbrack 0,1]$ and all $r \in \lbrack p,1)$
and, moreover, $\alpha \left( 0\right) \geq 0$ 
and $D_{0^{+}}^{q}\alpha \left( 1\right) \geq 0$;

$\omega^{2}\beta\left( t\right) 
- ^{C}D_{1^{-}}^{r}D_{0^{+}}^{q}\beta \left( t\right) 
-f\left( t,\beta \left( t\right) \right) \geq 0$
for all $t\in \lbrack 0,1]$ and all $r\in \lbrack p,1)$
and, moreover, $\beta \left( 0\right) \leq 0$ 
and $D_{0^{+}}^{q}\beta \left( 1\right) \leq 0$.

\noindent Functions $\alpha$ and $\beta$ are lower and upper solutions in
reverse order if $\alpha \left( t\right) \geq \beta \left( t\right)$, 
$0\leq t\leq 1$.
\end{definition}

\begin{remark}
\ If $\alpha $ and $\beta $ are, respectively, lower and upper 
solutions of problem $(P_1)$, then they are still lower 
and upper solutions for the sequence of problems generated 
by the boundary conditions \eqref{1.2}--\eqref{1.3} 
and the fractional differential equations obtained from \eqref{1.1} 
by replacing $p$ by $r$ for all $r \in \lbrack p,1)$.
\end{remark}


\section{\; Monotonicity for the Right Caputo Derivative}
\label{sec:3}

We begin by proving a useful monotonicity result 
for the right Caputo derivative. Theorem~\ref{thm:mon}
provides the right counterpart of the main result of \cite{7},
which was recently obtained for the left Caputo fractional 
derivative $^{C}D_{0^{+}}^{r}f\left( t\right)$.
It will be needed in the proof of our Lemma~\ref{lemma:uses:mon}.

\begin{theorem}
\label{thm:mon}
\ Assume that $f\in C^{1}[0,1]$ is such that $^{C}D_{1^{-}}^{r}
f\left(t\right) \geq 0$ for all $t\in \lbrack 0,1]$ and all $r \in (p,1)$ 
with some $p\in (0,1)$. Then $f$ is monotone decreasing. Similarly, 
if $^{C}D_{1^{-}}^{r}f\left( t\right) \leq 0$ for all $t$ and $r$
mentioned above, then $f$ is monotone increasing.
\end{theorem}

\begin{proof}
\ The proof is based on the following well-known propriety: 
\begin{equation*}
0\leq \underset{r \rightarrow 1^{-}}{\lim }^{C}D_{0^{+}}^{r}f\left(
t\right) =\underset{r\rightarrow 1^{-}}{\lim }
I_{0^{+}}^{1-r}f^{\prime }(t)=f^{\prime }(t)
\end{equation*}
(see Theorem~2.10 of \cite{8}). For the right Caputo fractional derivative 
$^{C}D_{1^{-}}^{r} f\left( t\right)$, one can prove the following 
analogue property: 
\begin{equation}
\label{eq:del}
0\leq \underset{r\rightarrow 1^{-}}{\lim }^{C}D_{1^{-}}^{r}f\left(
t\right) =\underset{r\rightarrow 1^{-}}{\lim }
-I_{1^{-}}^{1-r}f^{\prime }(t)=-f^{\prime }(t).
\end{equation}
Using \eqref{eq:del}, the proof follows in the same way as in \cite{7}.
\hfill $\square$
\end{proof}

\begin{remark}
\ Property \eqref{eq:del} and Theorem~\ref{thm:mon} can be obtained
straightforwardly from the results of \cite{8,7} by using the 
duality theory of Caputo--Torres between 
left and right fractional operators \cite{MyID:307}.
\end{remark}


\section{\; Existence of Solutions}
\label{sec:4}

First we solve a Riemann--Liouville fractional problem of order $q$:
\begin{equation}
\label{P2}
\tag{$P_2$}
\left\{ 
\begin{array}{l}
D_{0^{+}}^{q}u\left( t\right) =v\left( t\right),
\quad 0\leq t\leq 1, \\ 
u\left( 0\right) =0.
\end{array}
\right.
\end{equation}

\begin{lemma}
\ For $0<q<1$, the solution of problem \eqref{P2} is given by 
\begin{equation}
\label{3.1}
u\left( t\right) =\frac{1}{\Gamma \left( q-1\right) }\int_{0}^{t}\left(
t-s\right) ^{q-1}v\left( s\right) ds.  
\end{equation}
\end{lemma}

\begin{proof}
\ Applying the properties of the Riemann--Liouville fractional derivative
and the initial condition $u\left( 0\right) =0$, we get \eqref{3.1}.
\hfill $\square$
\end{proof}

Let $E:=C\left( \left[ 0,1\right] ,\mathbb{R}\right)$ be equipped with the
uniform norm $\left\vert \left\vert u\right\vert \right\vert =\underset{t\in 
\left[ 0,1\right] }{\max }\left\vert u\left( t\right) \right\vert$. Define
the operator $T$ on $E$ by
\begin{equation*}
Tv\left( t\right) =\frac{1}{\Gamma \left( q\right) }\int_{0}^{t}\left(
t-s\right) ^{q-1}v\left( s\right) ds=I_{0^{+}}^{q}v\left( t\right),
\quad t\in \left[ 0,1\right].  
\end{equation*}
Thus, $u\left( t\right) =Tv\left( t\right)$. Since 
$D_{0^{+}}^{q}u\left(1\right) =0$, problem $(P_1)$ 
is equivalent to the following Caputo
boundary value problem: 
\begin{equation}
\label{P3}
\tag{$P_3$}
\left\{ 
\begin{array}{l}
\omega^{2}Tv\left( t\right)
- {^{C}D_{1^{-}}^{p}}v\left( t\right) 
=f\left(t,Tv\left( t\right) \right), \quad 0\leq t\leq 1, \\ 
v\left( 1\right) =0.
\end{array}
\right.
\end{equation}
Let us make the following hypotheses:
\begin{description}
\item[$\left(H_1\right)$] 
\ there exists a nonnegative constant $A$ such that 
$$
\omega^{2}x-f\left( t,x\right) \leq A\left( 1-t\right) ^{1-r}
$$ 
for $0\leq t\leq 1$, $0\leq x\leq \frac{A}{\Gamma \left( q+1\right) }$, 
and for all $r \in \lbrack p,1)$;

\item[$\left(H_2\right)$] 
\ there exists a constant $B\leq 0$ such that 
$A\geq\left\vert B\right\vert$ and 
$$
\omega^{2}x-f\left( t,x\right) \geq B\left(1-t\right)^{1-r}
$$ 
for $0\leq t\leq 1,$ $\frac{B}{\Gamma \left(q+1\right)}
\leq x\leq 0$ and for $r \in \lbrack p,1)$.
\end{description}

\begin{lemma}
\ If hypotheses $(H_1)$ and $(H_2)$ hold, then 
problem $(P_1)$ has a lower and an upper solution.
\end{lemma}

\begin{proof}
\ Setting $\varphi \left( t\right) =A\left( 1-t\right)$, it follows that
\begin{equation*}
0\leq T\varphi \left( t\right) =I_{0^{+}}^{q}\varphi \left( t\right) 
=\frac{At^{q}\left( q+1-t\right) }{\Gamma \left( q+2\right) }
\leq \frac{A}{\Gamma\left( q+1\right) }.
\end{equation*}
Now we prove that $\alpha \left( t\right) =T\varphi \left( t\right) $ is an
upper solution of problem $(P_1)$. We have for all $r \in \lbrack p,1)$ that
\begin{equation*}
\begin{split}
\omega ^{2}T & \varphi \left(t\right)
- {^{C}D_{1^{-}}^{r}}\varphi \left( t\right)  
-f\left( t,T\varphi \left( t\right) \right) \\
&=\frac{-A}{\Gamma \left( 2-r\right) }\left( 1-t\right)
^{1-r}+\omega ^{2}T\varphi \left( t\right) 
-f\left( t,T\varphi \left(t\right) \right) \\
&\leq -A\left( 1-t\right) ^{1-r}+\omega ^{2}T\varphi \left( t\right)
-f\left( t,T\varphi \left( t\right) \right) \\
&\leq 0.
\end{split}
\end{equation*}
In addition, $\alpha \left( 0\right) =T\varphi \left( 0\right)=0$
and $D_{0^{+}}^{q}\alpha \left( 1\right) =\varphi \left( 1\right) =0$.
Thus, $\alpha \left( t\right) =T\varphi \left( t\right)$ 
is a lower solution of problem $(P_1)$.
Similarly, if we set $\psi \left( t\right) =B\left( 1-t\right)$, then 
$\beta \left( t\right) =T\psi \left( t\right)$ is an upper solution of
problem $(P_1)$.
\hfill $\square$
\end{proof}

\begin{lemma}
\ Under hypotheses $(H_1)$ and $(H_2)$, the upper and lower solutions 
of problem $(P_1)$ satisfy 
$\beta \left( t\right) \leq \alpha \left( t\right)$ 
and $D_{0^{+}}^{q}\beta \left( t\right) 
\leq D_{0^{+}}^{q}\alpha \left( t\right)$ 
for all $0\leq t\leq 1$.  
\end{lemma}

\begin{proof}
\ Since $\alpha \left( t\right) =T\varphi \left( t\right)$ 
and $\beta \left(t\right) =T\psi \left( t\right)$ are, respectively, 
lower and upper solutions of problem $(P_1)$, then from 
\begin{equation*}
\alpha \left( t\right) 
=\frac{A\left( q+1-t\right) t^{q}}{\Gamma\left(q+2\right) }\geq 0,
\quad
\beta \left( t\right) 
=\frac{B\left( q+1-t\right) t^{q}}{\Gamma \left(q+2\right) }\leq 0,
\end{equation*}
we get that
\begin{equation*}
D_{0^{+}}^{q}\alpha \left( t\right) =\varphi \left( t\right) =A\left(
1-t\right) \geq B\left( 1-t\right) =\psi \left( t\right) =D_{0^{+}}^{q}\beta
\left( t\right).
\end{equation*}
This completes the proof.
\hfill $\square$
\end{proof}

We consider a sequence of modified problems
\begin{equation}
\label{P4}
\tag{$(P_4)_{r}$} 
\left\{ 
\begin{array}{l}
-^{C}D_{1^{-}}^{r}v\left( t\right) 
=Fv\left( t\right), \quad 0\leq t\leq 1,\\ 
v\left( 1\right) =0
\end{array}
\right.
\end{equation}
for $r \in \lbrack p,1)$, where the operator 
$F:E\rightarrow E$ is defined by 
\begin{equation*}
Fv\left( t\right) 
=-\omega ^{2}T \min \left[ \varphi,
\max\left( v,\psi \right)\right]  
+f\left( t,T\min \left[
\varphi, \max \left( v,\psi \right) \right] \right),
\quad 0 \leq t\leq 1.
\end{equation*}

Next lemma gives the relation between the solution of a modified
problem \eqref{P4} and the solution of problem $(P_1)$.

\begin{lemma}
\label{lemma:uses:mon}
\ If $v$ is a solution of problem $\left( (P4)_{p}\right)$, 
then $u=Tv$ is solution of problem $(P_1)$ satisfying 
\begin{equation*}
\beta \left( t\right) \leq u\left( t\right) \leq \alpha \left( t\right)
\quad \text{and} \quad
D_{0^{+}}^{q}\beta \left( t\right) \leq D_{0^{+}}^{q}u\left( t\right) 
\leq D_{0^{+}}^{q}\alpha \left( t\right)
\end{equation*}
for all $0\leq t\leq 1$.
\end{lemma}

\begin{proof}
\ Firstly, for $r \in \lbrack p,1)$, we prove that if $v_{r}$ is a
solution of problem \eqref{P4}, then $\psi \left( t\right)
\leq v_{r}\left( t\right) \leq \varphi \left( t\right)$. Putting 
$\epsilon \left( t\right) =v_{r}\left( t\right) -\varphi \left( t\right)$,
and using the initial conditions $v_{r}\left( 1\right) =\varphi
\left( 1\right) =0$, it yields $\epsilon \left( 1\right) =0$. Suppose the
contrary, i.e., that there exists $t_{1}\in \left[ 0,1\right[$ such that 
$v_{r}\left( t_{1}\right) >\varphi \left( t_{1}\right)$. From the
continuity of $\epsilon $, we conclude that there exist $b\in \lbrack
t_{1},1)$ and $a\in \lbrack 0,t_{1}]$ such that $\epsilon \left( b\right) =0$
and $\epsilon \left( t\right) \geq 0$, $t\in \lbrack a,b]$. Applying the
right Caputo fractional derivative and taking into account the definition
of lower solution, we get 
\begin{equation*}
\begin{split}
^{C}D_{1^{-}}^{r}\epsilon \left( t\right)
&={^{C}D_{1^{-}}^{r}}v_{r}\left( t\right)
-{^{C}D_{1^{-}}^{r}}\varphi \left( t\right) \\
&=\omega ^{2}T \min\left[\varphi,\max\left( v_{r},\psi\right)\right] 
-f\left( t,T \min\left[\varphi,\max\left(v_{r},\psi\right)\right]\right)\\
&\qquad - {^{C}D_{1^{-}}^{r}} \, D_{0^{+}}^{q}\alpha \left( t\right) \\
&\leq 0
\end{split}
\end{equation*}
for $t\in \lbrack a,b]$. Thanks to Theorem~\ref{thm:mon}, 
we know that $\epsilon$ is increasing on $[a,b]$. Since 
$\epsilon \left( b\right) =0$, we conclude that $v_{r}\left( t\right)
\leq \varphi \left( t\right)$, $t\in \lbrack a,b]$, which leads to a
contradiction. Similarly, we prove that $\psi \left( t\right) \leq
v_{r}\left( t\right)$, $t\in \lbrack 0,1]$. From the above discussion,
if $v$ is a solution of problem $\left( (P4)_{p}\right)$, then 
\begin{equation*}
-^{C}D_{1^{-}}^{p}v\left( t\right) 
=\left( Fv\right) \left( t\right)
=-\omega ^{2}Tv\left( t\right) 
+f\left( t,Tv\left( t\right) \right).
\end{equation*}
Thus, $v$ is a solution of \eqref{P3} and, therefore, $u=Tv$ 
is a solution of $(P_1)$. Finally, 
the monotonicity of operator $T$ implies 
\begin{equation*}
T\psi \left( t\right) \leq Tv\left( t\right) \leq T\varphi \left( t\right),
\quad t\in \lbrack 0,1].
\end{equation*}
This achieves the proof.
\hfill $\square$
\end{proof}

Now we are ready to formulate and prove our main result 
of existence of solution for problem $(P_1)$.

\begin{theorem}
\label{thm:mr}
\ Assume that hypotheses $(H_1)$ and $(H_2)$ hold. Then, problem $(P_1)$ has at
least one solution $u$ such that 
$$
\beta \left( t\right) 
\leq u\left( t\right) \leq \alpha \left( t\right)
$$
and
$$
D_{0^{+}}^{q}\beta \left( t\right) 
\leq D_{0^{+}}^{q}u\left( t\right) 
\leq D_{0^{+}}^{q}\alpha \left( t\right)
$$
for all $0\leq t\leq 1$.
\end{theorem}

\begin{proof}
\ Define the operator $R$ on $E$ by 
$Rv\left( t\right) =I_{1^{-}}^{p}Fv\left( t\right)$,
$0\leq t\leq 1$. Set 
$$
\Omega :=\left\{v\in C\left[ 0,1\right] ,\psi \left( t\right) 
\leq v\left(t\right) \leq \varphi \left( t\right), 0\leq t\leq 1\right\},
$$
where 
\begin{equation*}
M:=\max \left\{ \left\vert \omega ^{2}x-f(t,x)\right\vert ,\beta \left(
t\right) \leq x\leq \alpha \left( t\right) ,0\leq t\leq 1\right\} .
\end{equation*}
Let $v\in \Omega$. Taking into account that 
$\beta \left( t\right) 
\leq T\left( \min \left[ \varphi,\max \left( v,\psi \right) 
\right] \right) \leq \alpha \left( t\right)$, then 
\begin{eqnarray*}
\left\vert Rv\left( t\right) \right\vert &\leq &I_{1^{-}}^{p}\left\vert
-\omega ^{2}T\left( \min \left[ \varphi ,\max \left( v,\psi \right)
\right] \right) +f\left( t,T\min \left[ \varphi ,\left( \max \left(
v,\psi \right) \right) \right] \right) \right\vert \\
&\leq &\frac{M}{\Gamma \left( p+1\right)}.
\end{eqnarray*}
Thus, $R\left( \Omega \right)$ is uniformly bounded and 
$R\left( \Omega\right) \subset \Omega $. For simplicity, denote 
\begin{equation*}
g\left( t\right) =-\omega ^{2}T\left( \min \left[ \varphi,
\max\left( v,\psi \right)\right] \right) +f\left( t,T\min \left[
\varphi,\max \left( v,\psi \right) \right] \right) .
\end{equation*}
For $0\leq t_{1}<t_{2}\leq 1$, we have 
\begin{equation*}
\begin{split}
\left\vert Rv\left( t_{1}\right) -Rv\left( t_{2}\right) \right\vert 
&\leq \left\vert I_{1^{-}}^{p}g\left( t_{1}\right) -I_{1^{-}}^{p}g\left(
t_{2}\right) \right\vert \\
&\leq \frac{1}{\Gamma \left( p\right) }\int_{t_{1}}^{t_{2}}\left(
s-t_{1}\right) ^{p-1}\left\vert g\left( s\right) \right\vert ds\\
&\quad + \frac{1}{\Gamma \left( p\right) }\int_{t_{2}}^{1}\left( \left(
s-t_{1}\right) ^{p-1}-\left( s-t_{2}\right) ^{p-1}\right) \left\vert g\left(
s\right) \right\vert ds \\
&\leq \frac{M}{\Gamma \left( p+1\right) }\left( \left( 1-t_{1}\right)
^{p}-\left( 1-t_{2}\right) ^{p}\right) \rightarrow 0
\text{ as } t_{1}\rightarrow t_{2}.
\end{split}
\end{equation*}
Therefore, $R\left( \Omega \right) $ is equicontinuous. We conclude by
the Arzela--Ascoli theorem that $R$ is completely continuous. Then, 
by Schauder's fixed point theorem, $R$ has a fixed point 
$v\in \Omega$. We conclude that $u=Tv$ is a solution of $(P_1)$ satisfying,
by Lemma~\ref{lemma:uses:mon}, 
$\beta \left(t\right) \leq u\left( t\right) \leq \alpha \left( t\right)$ 
and $D_{0^{+}}^{q}\beta \left( t\right) \leq D_{0^{+}}^{q}u\left( t\right) 
\leq D_{0^{+}}^{q}\alpha \left( t\right)$, $0\leq t\leq 1$. 
The proof is complete.
\hfill $\square$
\end{proof}


\section{\; An Illustrative Example}
\label{sec:5}

We present a simple example to illustrate our results. 
Consider problem $(P_1)$ with $\omega =1$, 
$p=q=\frac{1}{2}$, and
\begin{equation*}
f\left( t,x\right) =x-\frac{1}{100}\left( 1-t\right) ^{\frac{1}{2}},
\quad 0 \leq t\leq 1.
\end{equation*}
If we choose $A=\frac{1}{100}$ and $B=-\frac{1}{100}$, then we get 
$$
\omega^{2}x-f\left( t,x\right) =\frac{1}{100}\left( 1-t\right)^{\frac{1}{2}}
\leq A\left( 1-t\right) ^{1-r}
$$ 
and 
$$
\omega ^{2}x-f\left( t,x\right) 
=\frac{1}{100}\left( 1-t\right) ^{\frac{1}{2}}\geq 0
\geq B\left( 1-t\right) ^{1-q}
$$
for $0\leq t\leq 1$ and for all $r\in \lbrack p,1)$. 
Then all assumptions of Theorem~\ref{thm:mr} hold. Consequently, problem 
\begin{equation*}
\begin{gathered}
u\left( t\right) 
- {^{C}D_{1^{-}}^{1/2}} \, D_{0^{+}}^{1/2}u\left( t\right)
= u\left( t\right)-\frac{1}{100}\left( 1-t\right)^{\frac{1}{2}},
\quad 0\leq t\leq 1,\\
u\left( 0\right) =0, \quad D_{0^{+}}^{1/2}u\left( 1\right) =0,
\end{gathered} 
\end{equation*}
has a solution $u$ such that 
$\beta \left( t\right) \leq u\left( t\right) 
\leq \alpha \left(t\right)$. 
By direct computations we get
\begin{equation*}
\alpha\left(t\right) 
=\frac{A\left( q+1-t\right) t^{q}}{\Gamma \left(q+2\right)}
=\frac{t^{\frac{1}{2}}\left(\frac{3}{2}-t\right)}{100
\Gamma\left( \frac{5}{2}\right)} \geq 0,
\quad \beta \left( t\right) 
=-\frac{t^{\frac{1}{2}}\left( \frac{3}{2}-t\right) }{100
\Gamma\left(\frac{5}{2}\right)} \leq 0,
\end{equation*}
and 
$$
u\left( t\right) =\frac{t^{\frac{1}{2}}}{100}\left( 1-\frac{2t}{3}\right).
$$


\section*{\; Acknowledgments}

Guezane-Lakoud and Khaldi were supported by Algerian funds 
within CNEPRU projects  B01120120002 and B01120140061, 
respectively. Torres was supported by Portuguese funds 
through CIDMA and FCT, project UID/MAT/04106/2013,
and by AIMS-Cameroon.



\emergencystretch=\hsize

\begin{center}
\rule{6 cm}{0.02 cm}
\end{center}


\end{document}